\documentclass[10pt]{amsart}
\usepackage{amssymb,amsmath,txfonts,amsthm}
\usepackage{hyperref}
\usepackage{mathrsfs}

\newcommand\blfootnote[1]{%
  \begingroup
  \renewcommand\thefootnote{}\footnote{#1}%
  \addtocounter{footnote}{-1}%
  \endgroup
}
\newtheorem{theorem}{Theorem}
\newtheorem{prop}{Proposition}
\newtheorem{lemma}{Lemma}

\newtheorem{claim}{Claim}
\newtheorem{definition}{Definition}
\newtheorem{cor}{Corollary}
\newtheorem{conj}{Conjecture}
\numberwithin{equation}{section}

\def\Xint#1{\mathchoice
  {\XXint\displaystyle\textstyle{#1}}%
  {\XXint\textstyle\scriptstyle{#1}}%
  {\XXint\scriptstyle\scriptscriptstyle{#1}}%
  {\XXint\scriptscriptstyle\scriptscriptstyle{#1}}%
  \!\int}
\def\XXint#1#2#3{{\setbox0=\hbox{$#1{#2#3}{\int}$}
  \vcenter{\hbox{$#2#3$}}\kern-.5\wd0}}

\def\dashint{\Xint-}

\author{Gang Liu}
\address{School of Mathematical Sciences, Shanghai Key Laboratory of PMMP, East China Normal University}
\email{gliu@math.ecnu.edu.cn}

\title[Complete K\"ahler manifolds with nonnegative Ricci curvature]{Complete K\"ahler manifolds with nonnegative Ricci curvature}
\date{}
\begin{document}
\begin{abstract}
We consider complete K\"ahler manifolds with nonnegative Ricci curvature. The main results are: 
1. When the manifold has nonnegative bisectional curvature, we show that $\lim\limits_{r\to\infty}r^{2}\dashint_{B(p, r)}S$ exists. In other words, it depends only on the manifold. This solves a question of Ni in \cite{[Ni]}. 
Also, we establish estimates among volume growth ratio, integral of scalar curvature,  and the degree of polynomial growth holomorphic functions. The new point is that the estimates are sharp for \emph{any} prescribed volume growth rate. As a byproduct, we show that $\lim\limits_{r\to\infty}r^{2}\dashint_{B(p, r)}S<\epsilon$ iff the asymptotic volume ratio of the universal cover is almost maximal. 2. We discover a strong rigidity for complete Ricci flat K\"ahler metrics. Let $M^n (n\geq 2)$ be a complete K\"ahler manifold with nonnegative Ricci curvature and Euclidean volume growth. Assume either the curvature has quadratic decay, or the K\"ahler metric is $dd^c$-exact with quadratic decay of scalar curvature. If one tangent cone at infinity is Ricci flat, then $M$ is Ricci flat. In particular, the tangent cone is unique. In other words, \emph{we can test Ricci flatness of the manifold by checking one single tangent cone}. This seems unexpected, since apriori, there is no equation on $M$ and the Bishop-Gromov volume comparison is not sharp on Ricci flat (nonflat) manifolds. Such result is in sharp contrast to the Riemannian setting: Colding and Naber \cite{[CN]} showed that tangent cones are quite flexible when $Ric\geq 0$ and $|Rm|r^2<C$. This reveals subtle differences between Riemannian case and K\"ahler case. The result contains a lot of examples, such as all noncompact Ricci flat K\"ahler surfaces of Euclidean volume growth (hyper-K\"ahler ALE 4-manifolds classified by Kronheimer \cite{[Kr]}), higher dimensional examples of Tian-Yau type \cite{[TY2]}, as well as an example with irregular cross section \cite{[CH1]}. It also covers Ricci flat K\"ahler metrics of Euclidean volume growth on Stein manifolds with $b_2 = 0$, such as Ricci flat K\"ahler metrics on $\mathbb{C}^n$ \cite{[LiYang]}\cite{[Sz0]}\cite{[CR]}\cite{[Chiu]}. Note in this case, the cross section is singular.

\end{abstract}
\maketitle
\blfootnote{The author is partially supported by National Key Research and Development Program of China (No. 2022YFA1005502), NSFC No. 12071140, the New Cornerstone Science Foundation through the New Cornerstone Investigator Program
and the Xplorer Prize, Science and Technology Commission of Shanghai Municipality (No. 22DZ2229014)}

\section{Introduction}

In Riemannian geometry, it is interesting to study the distribution of scalar curvature on manifolds with nonnegative Ricci curvature. Conjectures of Yau \cite{[Yau]}, Gromov \cite{[Gromov]}, Naber \cite{[Naber]} predict bounds of scalar curvature integral. 
In this paper, we are interested in the opposite direction. 
We shall study the impact of scalar curvature integral on complete K\"ahler manifolds.  
It turns out there are very strong rigidity phenomena, under both $BK\geq 0$ (nonnegative bisectional curvature) and $Ric\geq 0$.

In \cite{[Ni]}, Ni proved a remarkable gap theorem on complete K\"ahler manifolds with nonnegative bisectional curvature:
\begin{theorem}\label{optgap}[optimal gap theorem]
Let $(M, p)$ be a complete noncompact K\"ahler manifold with nonnegative bisectional curvature. If $\lim\limits_{r\to\infty}\inf r^2\dashint_{B(p, r)}S = 0$, then $M$ is flat. Here $S$ is the scalar curvature, $\dashint$ is the average.
\end{theorem}

In the same paper, he asked whether or not $\lim\limits_{r\to\infty}r^2\dashint_{B(p, r)}S$ exists on complete K\"ahler manifolds with nonnegative bisectional curvature. Our first result answers Ni's question affirmatively:
\begin{theorem}\label{thmniquestion}
Let $(M, p)$ be a complete noncompact K\"ahler manifold with nonnegative bisectional curvature. Then $\lim\limits_{r\to\infty}r^2\dashint_{B(p, r)}S$ exists. In other words, such limit (possibly infinity) depends only on $M$.
\end{theorem}  
Through the works of \cite{[Ni1]}\cite{[NT2]}\cite{[L3]}\cite{[L5]}, we know that the following conditions are equivalent on a complete K\"ahler manifold with nonnegative bisectional curvature, if the universal cover does not split:
\begin{itemize}
\item $M$ has Euclidean volume growth;
\item There exists a nonconstant polynomial growth holomorphic function on $M$;
\item $\lim\limits_{r\to\infty}\sup r^2\dashint_{B(p, r)}S<+\infty$.

\end{itemize}
However, explicit relations seem ellusive from previous arguments.
 Our second result addresses this (it is clear that we just need to handle Euclidean volume growth case):
\begin{theorem}\label{thmsharpestimates}
Let $(M^n, p)$ be a complete noncompact K\"ahler manifold with nonnegative bisectional curvature and Euclidean volume growth. 
Let $d$ be the smallest degree of polynomial growth holomorphic functions (nonconstant) on $M$. Also let $$v =  \lim\limits_{r\to\infty}\frac{Vol(B(p, r)}{Vol(B_{\mathbb{C}^n}(0, r))},$$
$$\alpha = \lim\limits_{r\to\infty} r^2\dashint_{B(p, r)}S.$$
Then we have the following sharp estimates:

(1) $$4n^2(v^{-\frac{1}{n}}-1)\leq \alpha \leq 4n(v^{-1}-1),$$

(2) $$1\leq d\leq v^{-\frac{1}{n}},$$

(3) $$\lim\limits_{k\to\infty}\frac{dim(\mathcal{O}_k(M))}{k^n/n!} = v \leq 1 = \lim\limits_{k\to\infty}\frac{dim(\mathcal{O}_k(\mathbb{C}^n))}{k^n/n!}.$$

When the metric is unitary symmetric \cite{[WZ]}, the left side of $(1)$ and right side of $(2)$ become equality.  However, the metric is not necessarily unique in this case. The right side of $(1)$ becomes an equality if and only if $M$ splits off $\mathbb{C}^{n-1}$. 
\end{theorem}

For previous sharp estimates on K\"ahler manifolds with nonnegative bisectional curvature, the equality is only valid for Euclidean or splitting case. In our case, sharp equalities may hold under \emph{any} asymptotic volume ratio, while the metric does not split. 
$(3)$ indicates a connection between (asymptotic) dimension estimate (compare \cite{[Ni1]}\cite{[CFLZ]}) and Bishop-Gromov volume comparison. These estimates support the common sense in a quantitative way: the larger the curvature, the smaller the volume.
Another interesting thing is that we obtain volume ratio lower bound, when the integral of scalar curvature has an upper bound. This can be regarded as an inverse of volume comparison. Such thing is not possible if we only assume $Ric\geq 0$. As a byproduct, we obtain a quantitative version of Ni's optimal gap theorem:
\begin{cor}\label{niquantitative}
Let $(M^n, p)$ be a complete noncompact K\"ahler manifold with nonnegative bisectional curvature. Let $(\tilde{M},\tilde{p})$ be the universal cover of $M$. Then the smallness of
$\lim\limits_{r\to\infty} r^2\dashint_{B(p, r)}S$ is equivalent to the smallness of $1 - \lim\limits_{r\to\infty}\frac{vol(B(\tilde{p}, r))}{\omega_{2n}r^{2n}}$, where $\omega_{2n}$ is the volume of unit ball in $\mathbb{C}^n$.

\end{cor}

Now we move to $Ric\geq 0$. 
A result of Anderson \cite{[A1]} states that if the asymptotic volume ratio is sufficiently close to that of the Euclidean space, then Ricci flatness is the same as flatness. We have the following generalization of Anderson's result in the K\"ahler setting:

\begin{theorem}\label{bigvolume}
There exists $\epsilon = \epsilon(n)>0$ such that the following hold: Assume $M^n$ is a complete noncompact K\"ahler manifold with nonnegative Ricci curvature and $$v = \lim\limits_{r\to\infty}\frac{Vol(B(p, r)}{Vol(B_{\mathbb{C}^n}(o, r))}\geq 1-\epsilon.$$ If
$$\lim\limits_{r\to\infty}\inf r^2\dashint_{B(p, r)}S=0$$
Then $M$ is flat.

\end{theorem}

It is clear that such result cannot be true, if the volume is not close to Euclidean, e.g., the Eguchi-Hanson metric. Unlike the flat (Euclidean) case, there is abundance of complete Ricci flat K\"ahler manifolds, even assuming Euclidean volume growth and quadratic curvature decay. The reader is referred to \cite{[CH2]} for a description of these manifolds.

\begin{definition}
Let $(M^n, \omega) (n\geq 2)$ be a complete noncompact K\"ahler manifold with nonnegative Ricci curvature and Euclidean volume growth. We say $M$ satisfies condition $A$, if $|Rm|\leq \frac{C}{r^2}$; $M$ satisfies condition $B$, if $S\leq \frac{C}{r^2}$ and $\omega$ is $dd^c$-exact.
\end{definition}

Frequently, we shall consider tangent cones of complete manifolds. The complex structure on a tangent cone is induced from the sequence of rescaled manifolds. Below, when we talk about convergence, we mean both the metric (Gromov-Hausdorff distance) and the complex structure in the regular part. 
\begin{definition}
Let $(M^n, \omega) (n\geq 2)$ be a complete noncompact K\"ahler manifold with nonnegative Ricci curvature and Euclidean volume growth. Let $V$ be a tangent cone of $M$ at infinity. We say $V$ is Ricci flat, if the metric on $V$ is Ricci flat (hence smooth) away from a real codimension $4$ set, and there exists a nontrival parallel pluri-canonical form on the metric regular part.
\end{definition}

Note if $M$ is a complete Ricci flat K\"ahler manifold with Euclidean volume growth, then by Cheeger-Naber \cite{[CN]}, any tangent cone is Ricci flat. 
\medskip

The interplay between the space (manifold) and its asymptotic cones (tangent cone) is fundamental in geometric analysis. On the one hand, uniqueness of asymptotic cones is expected when there are natural equations (e.g., Einstein manifolds, minimal surfaces, solitons). On the other hand, one looks for classification of spaces, according to their asymptotic cones.
The following theorem asserts that when a K\"ahler manifold has nonnegative Ricci curvature, one tangent cone at infinity might impose very strong restriction on the manifold (compare Colding \cite{[Coldingvolume]}): 

\begin{theorem}\label{ricflatmain1}
Let $M^n$ satisfy either condition $A$ or condition $B$. If one tangent cone at infinity is Ricci flat, then $M$ is Ricci flat. Hence the tangent cone is unique.
\end{theorem}
Such result seems unexpected, since apriori, there is no equation on $M$, and Bishop-Gromov volume comparison is not sharp on Ricci flat (nonflat) manifolds. 
This seems to be the \emph{first} rigidity result for noncompact Ricci flat metrics, from metric geometry view point.  Note by Theorem $B$ in \cite{[HL]}, such result is not true, if we merely assume the scalar curvature is nonnegative. The result is in sharp contrast to the Riemannian setting: according to Theorem $1.1$ in \cite{[CN]} by Colding-Naber, tangent cones are quite flexible (in Colding-Naber case, one can make $Ric\geq 0$ and $|Rm|r^2\leq C$). This reveals subtle differences between Riemannian case and K\"ahler case. Theorem \ref{ricflatmain1} covers all noncompact Ricci flat K\"ahler surfaces of Euclidean volume growth (hyper-K\"ahler ALE 4-manifolds classified by Kronheimer \cite{[Kr]}). It also covers many examples of Tian-Yau type \cite{[TY1]}\cite{[TY2]}\cite{[Candelas]}\cite{[Stenzel]}\cite{[Bando]}\cite{[CH0]}\cite{[Goto]}\cite{[Joyce]}\cite{[Coevering]} in higher dimensions, as well as an example with irregular cross section \cite{[CH1]}. Also, it covers complete Ricci flat K\"ahler metrics with Euclidean volume growth on $\mathbb{C}^n$, e.g., \cite{[LiYang]}\cite{[Sz0]}\cite{[CR]}\cite{[Chiu]}. Note in this case, the cross section is singular. We also note that given a tangent cone, the Ricci flat metric on the manifold need not be unique \cite{[Chiu]}.

\begin{cor}\label{ricflatstein}
Let $(M^n, \omega) (n\geq 2)$ be a complete noncompact K\"ahler manifold with nonnegative Ricci curvature, $S\leq \frac{C}{r^2}$ and Euclidean volume growth. Assume $H^1(M, \mathcal{O}) = 0$ (e.g., $M$ is Stein) and $b_2 = 0$. If one tangent cone at infinity is Ricci flat, then $M$ is Ricci flat.

\end{cor}

We also have results on local tangent cones:
\begin{cor}\label{ricflatlocal1}
Let $(M_i, p_i)$ be a sequence of complete K\"ahler manifolds with $Ric_{M_i} \geq 0$, $vol(B(p_i, 1))>v$. Assume there exists $C, v>0$ such that $|Rm|_{M_i}(x)d^2(x, p_i)\leq C$ for all $i$ and $d(x, p_i)<1$ . Let $(M_i, p_i)$ converge in the pointed Gromov-Hausdorff sense to $(X, o)$. If one tangent cone at $o\in X$ is Ricci flat, then all tangent cones at $o$ are Ricci flat, and the tangent cone is unique (certainly $X$ need not be Ricci flat).
\end{cor}
\begin{cor}\label{ricflatlocal2}
Let $(M_i, p_i)$ be a sequence of polarized K\"ahler manifolds with $Ric_{M_i} \geq -1$, $vol(B(p_i, 1))>v$. Assume there exists $C, v>0$ such that $S_{M_i}(x)d^2(x, p_i)\leq C$ for all $i$ and $d(x, p_i)<1$ . Assume $(M_i, p_i)$ converges in the pointed Gromov-Hausdorff sense to $(X, o)$. If one tangent cone at $o\in X$ is Ricci flat, then all tangent cones at $o$ are Ricci flat, and the tangent cone is unique.
\end{cor}

The uniqueness of tangent cones is an important problem in geometric analysis. A well-known conjecture states that tangent cones are unique for noncollapsed Einstein manifolds (see \cite{[CT]}\cite{[CM]}\cite{[DS]} for important progress). However, without assuming Einstein equation, the uniqueness is not ensured in general \cite{[Perelman]}\cite{[CC1]}\cite{[CN]}.
In our settings, tangent cones are unique, as long as one tangent cone is Ricci flat. As we mentioned before, by Colding-Naber's contruction \cite{[CN]}, such thing cannot be true in the Riemannian setting.
\begin{cor}\label{ricflatcorintegral}
Let $M^n$ satisfy condition $A$. If $\lim\limits_{r\to\infty}\inf r^2\dashint_{B(p, r)}S = 0$, then $M$ is Ricci flat.

\end{cor}
\begin{conj}\label{conj1}
Let $M$ be a complete noncompact K\"ahler manifold with nonnegative Ricci curvature and Euclidean volume growth. Assume one tangent cone at infinity is Ricci flat, then $M$ is Ricci flat. Hence, the tangent cone should be unique. We can also ask about the similar question on local tangent cones.
\end{conj}
One can compare Conjecture \ref{conj1} with a theorem of Colding \cite{[Coldingvolume]}, which says that if one tangent cone is Euclidean, then the manifold is Euclidean.

\begin{conj}\label{conj2}
Let $(M^n, p)$ be a complete noncompact K\"ahler manifold with nonnegative Ricci curvature. If $\lim\limits_{r\to\infty}\inf r^2\dashint_{B(p, r)}S = 0$, then $M$ is Ricci flat. We can also ask whether or not $\lim\limits_{r\to\infty} r^2\dashint_{B(p, r)}S$ exists (possibly infinity).
\end{conj}
\begin{cor}\label{ricflatcorpointwise}
Let $M^n$ satisfy condition $A$. If the scalar curvature has faster than quadratic decay, then $M$ is Ricci flat.

\end{cor}

This can be regarded as a gap theorem for Ricci flat metrics. In fact, one only needs to assume that scalar curvature has faster than quadratic decay along a sequence of annuli with fixed ratio.

\begin{cor}\label{ricflatcorintegralsn}
Let $M^n$ satisfy condition $A$. If $\int_{B(p, r)}S^n =o(\log r)$, then $M$ is Ricci flat.

\end{cor}

Both the power and $o(\log r)$ are sharp.

\begin{cor}\label{ricflatcorrmn}
Let $M^n$ satisfy condition $A$. If $\int_{B(p, r)}|Rm|^n =o(\log r)$, then $M$ is Ricci flat.

\end{cor}

In this case, the tangent cone is an orbifold with flat metric. By \cite{[SunZhang]}, the metric converges in polynomial rate to the tangent cone. Thus, curvature decays at a polynomial rate greater than $2$. Thus $\int_M|Rm|^n <+\infty$.
So Corollary \ref{ricflatcorrmn} can also be regarded as a gap result. These manifolds have been studied in \cite{[BKN]}\cite{[T1]}\cite{[Joyce]}\cite{[HRI]}.

Now we describe the strategies in the proof. For $BK\geq 0$, recall in \cite{[L5]}, we used Gromov compactness to derive the existence of polynomial growth holomorphic functions. However, the method seems too soft to achieve sharp estimates. The new point is to use Cheeger-Colding-Tian \cite{[Coldingvolume]}\cite{[CC1]}\cite{[CCT]}, \cite{[LT]}\cite{[Lo]} that the tangent cone is biholomorphic to $\mathbb{C}^n$, as well as the three circles lemma \cite{[L1]}. The key thing is to express the volume ratio and scalar curvature integral in terms of degrees of polynomial growth holomorphic functions. Thus the degrees are deeper underlying invariants of the manifold. In the course of the proof, we have to overcome the difficulty on non-smoothness of the metric on tangent cones. When these connections are built, the proof for $BK\geq 0$ follows by arithmetic-geometric inequality. For $Ric\geq 0$, a difficulty is that usual monotonicity formulae (e.g., Bishop-Gromov volume comparison) are not sharp for Ricci flat (nonflat) metrics. We use a new monotonicity, $\int Ric^n$. The important thing is that such quantity is rescale invariant. Also, it behaves well under Gromov-Hausdorff convergence (under certain assumptions) and provides connections between different tangent cones. With such aid, we are able to show that if one tangent cone is Ricci flat, then all tangent cones are Ricci flat. Then holomorphic spectra for Ricci flat tangent cones are algebraic numbers \cite{[MSY]}\cite{[DS]}, hence rigid. We apply algebraic compactification in \cite{[DS]}\cite{[L2]}. 
Eventually, we just need to test the integral of Ricci form (use scale invariance of Ric) along an algebraic curve. If $M$ is not Ricci flat, integration by parts gives a contradiction.

\medskip

Notation: Unless otherwise stated, $\Psi(x_1, .., x_n)$ means quantities converging to zero, when $x_1, .., x_n$ are converging to zero. We also let $dd^c = \sqrt{-1}\partial\overline\partial$.

\medskip

\begin{center}
\bf  {\quad Acknowledgments}
\end{center}
The author thanks Professors A. Naber, L. Ni, G. Tian for interests and helpful discussions. He also thanks Professors J. Lott and J.-P Wang for answering several questions.

\section {The case of nonnegative bisectional curvature}

Proof of Theorem \ref{thmniquestion}:

Note $n=1$ case follows from the usual Cohn-Vossen inequality \cite{[CV]}. Below we let $n\geq 2$. We claim that if the limit is independent of $r$, then it is independent of the base point.
Assume $\lim\limits_{r\to\infty} r^2\dashint_{B(p_1, r)}S = a$ (possibly infinity). 
For any $p_2\in M$, let $b=d(p_1, p_2)$. Then 

 $$\int_{B(p_1, r-b)}S \leq \int_{B(p_2, r)}S \leq \int_{B(p_1, r+b)}S.$$
Also by Bishop-Gromov volume comparsion, 
$$1<\frac{vol(B(p_1, r+b))}{vol(B(p_1, r-b))}\leq (\frac{r+b}{r-b})^{2n}.$$
Putting these two inequalities together, we obtain that $$\lim\limits_{r\to\infty} r^2\dashint_{B(p_2, r)}S = a.$$
Below we fix the base point $p$.
There are two cases:
\begin{itemize}
\item $\lim\limits_{r\to\infty}\sup r^2\dashint_{B(p, r)}S = \infty$;
\item $\lim\limits_{r\to\infty}\sup r^2\dashint_{B(p, r)}S < \infty$.
\end{itemize}
In the first case, we can apply ~\cite[Theorem $4.1$]{[Ni]} (take $\rho = Ric$) to conclude that 
$$\lim\limits_{r\to\infty} r^2\dashint_{B(p, r)}S = \infty.$$  Thus we are left with the second case.
Let us first assume that the universal cover does not split off a flat $\mathbb{C}$ factor. It is known from ~\cite[Theorem $2$]{[L3]} that $M$ has Euclidean volume growth.

Consider a sequence $r_i\to\infty$, assume $(M_i, p_i, d_i) = (M, p, \frac{d}{r_i})$ converges in the pointed-Gromov-Hausdorff sense to a tangent cone $(V, o)$. As is known by \cite{[LT]} ~\cite[Proposition $6.1$]{[Lo]}, $V$ is smooth in the complex analytic sense and $V$ is biholomorphic to $\mathbb{C}^n$. Moreover, there exists a holomorphic coordinate $(z_1, .., z_n)$ on $V$ and a reeb vector field given by $J(r\frac{\partial}{\partial r}) = \sum Re(\sqrt{-1}d_jz_j\frac{\partial}{\partial z_j})$, where $d_j\geq 1$ corresponds to weighted degree of the homogeneous function $z_j$.

Under the pointed Gromov-Hausdorff convergence $(M_i, p_i)\to (V, o)$, we can find holomorphic functions on $M_i$ converging uniformly on each compact set to $z_j$ on $V$. Such functions serve as local holomorphic coordinates around $p_i$. For simplicity, we still call these functions $z_1, .., z_n$, defined on $M_i$. In other words, over $M_i$ and $V$, we have the same holomorphic coordinate (different metrics). Let the rescaled K\"ahler form on $M_i$ be $\omega_i$. As $Ric\geq 0$, $\log |dz_1\wedge dz_2\wedge\cdot\cdot\cdot\wedge dz_n|_{\omega_i}^2$ is psh. Hence by passing to a subsequence, we may assume it converges to a psh function on $V$ (we use the same holomorphic chart). Let us call it $\log |dz|^2$. 

\begin{claim}\label{clhomogeneous}
Under the homothetic vector field $r\frac{\partial}{\partial r}$, $|dz|^2=e^{\log|dz|^2}$ is homogeneous of degree $\sum d_j - n$. 
\end{claim}
\begin{proof}
Let $\varphi_i=\log|dz_1\wedge\cdot\cdot\wedge dz_n|_{\omega_i}^2, \varphi=\log |dz|^2$. Then $\varphi_i\to\varphi$ in $L_{loc}^1$ sense. Pick an arbitrary bounded domain $\Omega\subset V$. Under the same holomorphic coordinate $(z_1, .., z_n)$, we take $\Omega_i = \Omega$, where $\Omega_i\subset M_i$. Let $dz\wedge d\overline{z}$ be the Euclidean measure on $\mathbb{C}^n$.

\begin{lemma}\label{volumesame}
$vol(\Omega)=\int_{\Omega}e^{-\varphi}dz\wedge d\overline{z}$.
\end{lemma}
\begin{proof}
As $\varphi_i\to\varphi$ in $L_{loc}^1$, $e^{-\varphi_i}\to e^{-\varphi}$ almost everywhere. Note $$vol(\Omega_i) = \int_{\Omega_i}e^{-\varphi_i}dz\wedge d\overline{z}.$$ By Cheeger-Colding \cite{[CC2]}, $$vol(\Omega_i)\to vol(\Omega).$$ Then 
$$\int_{\Omega}e^{-\varphi}dz\wedge d\overline{z}\leq Vol(\Omega).$$ Assume strict inequality holds. This occurs only when $$\lim\limits_{i\to\infty}\inf\int_{\Omega_i\cap \{\varphi_i\leq -C\}}e^{-\varphi_i}dz\wedge d\overline{z}\geq \epsilon_0,$$ where $\epsilon_0>0$ is fixed and $C$ is arbitrarily large. Then we can find a sequence $C_i\to\infty$ slowly such that $$E_i = \Omega_i\cap \{\varphi_i\leq -C_i\}$$ has measure $$m(E_i)=vol(E_i)\geq\epsilon_0,$$ for all large $i$. 
By passing to a subsequence, we may assume $E_i$ converges to $E\subset\Omega$ in the weak sense. Then $m(E)\geq\epsilon_0$. Note regular points on $\Omega$ has full Hausdorff measure \cite{[CC2]}. Then there exists a regular point $x\in\Omega$ and a sequence $c_k\to 0$ such that $$\frac{m(E\cap B(x, c_k))}{m(B(x, c_k))}\geq 1/2.$$ Recall by ~\cite[Theorem $1.5$]{[LS1]}, around $x$, we have a holomorphic chart $(w^k_1, .., w^k_n)$ on $B(x, c_k)$ which is a $\Psi(\frac{1}{k})c_k$-Gromov-Hausdorff approximation to its image in $\mathbb{C}^n$. We can pull back the holomorphic coordinate to $M_i$, say $x_i\to x$, $w^{ki}_s\to w^k_s$ (fix $k$). Then \begin{equation}\label{E_imeasure}\frac{m(E_i\cap B(x_i, c_k))}{m(B(x_i, c_k))}\geq 1/3.\end{equation} By Cheeger-Colding \cite{[CC1]}, we have 
$$\dashint_{B(x_i, c_k)}|\langle dw^{ki}_s, d\overline{w}^{ki}_t\rangle - \delta_{st}|^2<\Psi(\frac{1}{k}).$$ Let $$E'_{i, k} = \{y\in B(x_i, c_k)||dw^{ki}_1\wedge dw^{ki}_2 \wedge \cdot\cdot\cdot\wedge dw^{ki}_n(y)|>0.5\}.$$
So if $k$ is large enough (fixed), then for all large $i$, \begin{equation}\label{E'lowerbound}m(E'_{i, k})>0.9m(x_i, c_k).\end{equation} Note the holomorphic chart $(w^{ki}_1, ..., w^{ki}_n)$ has fixed size. Hence by Cauchy estimate, on $E'_{i, k}$, $$|dz_1\wedge dz_2\wedge\cdot\cdot\cdot\wedge dz_n|>\frac{1}{C(k)}.$$ If $i$ is large enough, there is a contradiction between (\ref{E_imeasure}) and (\ref{E'lowerbound}).

\end{proof}

Let $\alpha>0$ be arbitrary. Let $\Omega^\alpha$ be the homothetic action of $\Omega$ (induced by $r\frac{\partial}{\partial r}$) by multiplying $\alpha$. Let $F_\alpha$ be such homothetic action.
Then by Lemma \ref{volumesame}, $$vol(\Omega^\alpha)=\int_{\Omega^\alpha}e^{-\varphi}dz\wedge d\overline{z}=\alpha^n\int_{\Omega}e^{-\varphi}dz\wedge d\overline{z}.$$ Therefore, by pulling back coordinate via the biholomorphism $F_\alpha$, we obtain that $$\int_{\Omega}e^{-\varphi}dz\wedge d\overline{z}=\alpha^{\sum d_j - n}\int_{\Omega}e^{-F_\alpha^*(\varphi)}dz\wedge d\overline{z}.$$ That is to say, $$\int_{\Omega}(e^{-\varphi}-\alpha^{\sum d_j - n}e^{-F_\alpha^*(\varphi)})dz\wedge d\overline{z}=0.$$

Since $\alpha$ and $\Omega$ are arbitrary, we concludes the proof of Claim \ref{clhomogeneous}.
\end{proof}

\begin{prop}\label{proplogpsh}
$\psi = \log r^2$ is psh on $V$.
\end{prop}
\begin{proof}
By \cite{[DS]} and $(2.4)$ in \cite{[Lcv]}, we can consider approximation of $r^2$ by $\rho_i$ on $M_i$, where $dd^c\rho_i = \omega_i$ on $B(p_i, 10)$. 
By Bochner formula, $$\Delta (|\nabla\rho_i|^2 - 4\rho_i)\geq 0.$$
Let $\hat{u} = \max(|\nabla\rho_i|^2 - 4\rho_i, 0)$. Then $\hat{u}$ is subharmonic. Furthermore, by Cheeer-Colding \cite{[CC1]}, $$\int_{B(p_i, 9)} \hat{u}^2\leq\int_{B(p_i, 9)} (|\nabla\rho_i|^2 - 4\rho_i)^2<\Psi(\frac{1}{i}).$$ Then by mean value inequality, on $B(p_i, 5)$, 
$$|\nabla \rho_i|\leq 2r+\Psi(\frac{1}{i}).$$
Thus $$dd^c \log\rho_i = \frac{\rho_idd^c\rho_i - d\rho_i\wedge d^c\rho_i}{\rho_i^2}\geq -\Psi(\frac{1}{i})\omega_i.$$
The proposition follows by letting $i\to\infty$.

\end{proof}
Let $u$ be a smooth function on $V\backslash\{o\}$ which defines a smooth K\"ahler cone metric $\omega_0$ ($\omega_0=dd^c u$) with reeb vector $J(r\frac{\partial}{\partial r}) = \sum Re(\sqrt{-1}d_jz_j\frac{\partial}{\partial z_j})$. We also require that $u$ is invariant under reeb. Such $u$ exists by \cite{[MSY]}.
We would like to approximate the limit metric on $V$ by a family of smooth K\"ahler cone metrics in the potential sense, without changing reeb. The argument is similar to ~\cite[Lemma $3.1$]{[HLi]}. For reader's convenience, we provide the details. Below $\tau(k)$ represents a positive sequence converging to zero. The value might change from line to line.
Now consider an annulus in $V$ which is a complex manifold, say $A=B(o, C)\backslash B(o, 1/C)$, where $C$ is large so that $A$ contains the level set $\{u=1\}$. By \cite{[BK]} and Proposition \ref{clhomogeneous}, we can approximate $\psi$ by $\psi_k$ on $A$, where $\psi_k$ is almost psh in the sense that $dd^c\psi_k\geq -\tau(k)\omega_0$. Now by averaging torus integration (use invariant measure) generated by the reeb, we find that \begin{equation}\label{ddcpsik}dd^c\psi_k\geq -\tau(k)\omega_0\end{equation}
still holds.
Let $'$ be the derivative along $r\frac{\partial}{\partial r}$.
\begin{claim}\label{clpsik}
 $(\psi'_k)'\geq -\tau(k)$.
\end{claim}
\begin{proof}
This follows by considering $dd^c\psi_k$ along the holomorphic direction generated by reeb, using that $\psi_k$ is reeb invariant, $\frac{\partial}{\partial log r} = r\frac{\partial}{\partial r}$ and $\log|z| +i\theta$ as a local chart on the holomorphic curve. 
\end{proof}

Recall $\psi_k\to \log r^2$ pointwise. By Claim \ref{clpsik} and an ode argument, we find that \begin{equation}\label{psikcloseto2}|\psi'_k - 2|<\tau(k).\end{equation}  Define $$H_k = \psi_k - \log u.$$ Then $H_k'$ is almost 0.  Let us define a function $h_k$ which is constant along $r\frac{\partial}{\partial r}$ and the value equals $H_k$ on the level set $\Sigma$ where $u=1$. Given an arbitrary $p\in \Sigma$, we can find a local holomorphic chart $(w_1, w_2, ..., w_n)$ such that $\frac{\partial}{\partial w_1} =r\frac{\partial}{\partial r} - \sqrt{-1}Jr\frac{\partial}{\partial r}$. Hence $h_k$ is constant on each $w_1$-slice. We may adjust the coordinate (without changing $\frac{\partial}{\partial w_1}$) such that $\frac{\partial}{\partial w_1}$ is orthogonal to $\frac{\partial}{\partial w_i} (i\geq 2)$ at $p$, with repect to the K\"ahler cone metric $\omega_0$. Let $e\in T_p^{1, 0}V$, $e=a+b$ with $|e|_{\omega_0} = 1$, where $a\in $ span $\{\frac{\partial}{\partial w_1}\}$, $b\in $ span $\{ \frac{\partial}{\partial w_2}, ..., \frac{\partial}{\partial w_n}\}$. As $h_k$ is constant along $\frac{\partial}{\partial w_1}, \frac{\partial}{\partial \overline{w_1}}$, $$(h_k)_{e\overline{e}}=(h_k)_{b\overline{b}}= (h_k)_{b\overline{b}}|_\Sigma.$$
By using (\ref{ddcpsik}), (\ref{psikcloseto2}) and $(\log u)' = 2$, we find 
 \[ \begin{aligned}
     (h_k)_{b\overline{b}}|_\Sigma &= (\psi_k)_{b\overline{b}}|_\Sigma - (\log u)_{b\overline{b}}|_\Sigma \\&\geq (\psi_k)_{b\overline{b}} - (\log u)_{b\overline{b}} - \tau(k)\\&\geq - (\log u)_{b\overline{b}} -\tau(k)\\& =  - (\log u)_{e\overline{e}} -\tau(k).
\end{aligned}
   \]

In other words, on $\Sigma$, we have $$dd^c(h_k+\log u)\geq -\tau(k)\omega_0.$$ Thus, on the local leaf of reeb foliation, one can find constants $s_k\to 1$ such that $$d_Bd_B^c(s_kh_k) + \omega^T>0,$$ where $d_B$ is the exterior differential on local leaf space, $\omega^T$ is the metric on local leaf space induced by $\omega_0$. So $u_k = \exp(s_kh_k + \log u)$ defines a smooth K\"ahler cone metric with the same reeb as $\omega_0$. Also the potential converges to $r^2$.

\medskip

Define $$\lambda=\sum d_j - n.$$ Set $1>a>0$ as a constant. Let $\chi_k$ be a smooth nonnnegative cut-off function on $\mathbb{R}^+$ such that $\chi_k = 1$ on $[a^2, 1]$ and the supported in $(a^2-1/k, 1+1/k)$.
Let $S_i$ be the scalar curvature on $M_i$, $\omega = dd^c r^2$ on $(V, o)$. By Chern-Levine-Nirenberg ~\cite[pp. 146]{[De2]},  \[ \begin{aligned}r_i^2\dashint_{B(p, r_i)} S &= \dashint_{B(p_i, 1)}S_i \\&=\dashint_{B(p_i, 1)}dd^c\log |dz_1\wedge dz_2\wedge\cdot\cdot\cdot\wedge dz_n|^2\wedge\omega_i^{n-1}\\&\to \dashint_{B(o, 1)}dd^c\log |dz|^2\wedge\omega^{n-1}.\end{aligned}
   \]
Let $\tilde{\omega}_k$ be a sequence of smooth K\"ahler cone metrics over $(V, o)$ (same reeb) so that the potential $\rho_k$ is converging to $r^2$.
Let us evaluate the integral over annulus $A(o, 1, a) = B(o, 1)\backslash \overline{B(o, a)}$.
 \[ \begin{aligned}\frac{1}{vol(B_V(o, 1))}\int_{A(o, 1, a)}dd^c\log |dz|^2\wedge\omega^{n-1} &= \lim\limits_{k\to\infty}\dashint_{B(o, 1)}\chi_k(\rho_k)dd^c\log |dz|^2\wedge\tilde{\omega}_k^{n-1}
\\&= \lim\limits_{k\to\infty} \dashint_{B(o, 1)} \log |dz|^2dd^c\chi_k(\rho_k)\wedge\tilde{\omega}_k^{n-1}
\\&= \lim\limits_{k\to\infty} \dashint_{B(o, 1)} (\log |dz|^2-\lambda\log\rho_k)dd^c\chi_k(\rho_k)\wedge\tilde{\omega}_k^{n-1}\\&+\lim\limits_{k\to\infty} \dashint_{B(o, 1)}\lambda\log\rho_kdd^c\chi_k(\rho_k) \wedge\tilde{\omega}_k^{n-1}. \end{aligned}
   \]
As $\tilde{\omega}_k$ is smooth, by direct computation, we find $$dd^c\chi_k(\rho_k)\wedge\tilde{\omega}_k^{n-1} = F_k(\rho_k)\tilde{\omega}_k^n = G_k(\rho_k)ds_kdr_k,$$ where in the last step, we used polar coordinate. 
Then \[ \begin{aligned}&\lim\limits_{k\to\infty} \dashint_{B(o, 1)}(\log |dz|^2-\lambda\log\rho_k)dd^c\chi_k(\rho_k)\wedge\tilde{\omega}_k^{n-1}\\& = \lim\limits_{k\to\infty} \dashint_{B(o, 1)} (\log |dz|^2-\lambda\log\rho_k) G_k(\rho_k)dr_kds_k\\&=0 , \end{aligned} \] where we used that $\log |dz|^2-\lambda\log\rho_i$ is constant along $r\frac{\partial}{\partial r}$ (ensured by Claim \ref{clhomogeneous}) and the Fubini theorem.
As $\tilde{\omega}_k$ is a smooth K\"ahler cone metric, by direct computation, we find that
$$\lim\limits_{k\to\infty}\dashint_{B(o, 1)}\lambda\log\rho_kdd^c\chi_k(\rho_k) \wedge\tilde{\omega}_k^{n-1} = \lambda C(n, a),$$ where $C(n, a)$ is an explicit number. 
Then  \[ \begin{aligned}\int_{B(o, 1)}dd^c\log |dz|^2\wedge\omega^{n-1} &= \int_{A(o, 1, a)}dd^c\log |dz|^2\wedge\omega^{n-1} + \int_{B(o,a)}dd^c\log |dz|^2\wedge\omega^{n-1}\\& = C(n, a)\lambda vol(B(o, 1)) + \int_{B(o,a)}dd^c\log |dz|^2\wedge\omega^{n-1}.\end{aligned}
   \]
By the assumption $$\lim\limits_{r\to\infty}\sup r^2\dashint_{B(p, r)}S < \infty,$$ we have $$\int_{B(o,a)}dd^c\log |dz|^2\wedge\omega^{n-1}\leq Ca^{2n-2}.$$ As $n\geq 2$, by letting $a\to 0$, we find that 
\begin{equation}\label{eqforintegral}\dashint_{B(o, 1)}dd^c\log |dz|^2\wedge\omega^{n-1} = C(n)\lambda = C(n)(\sum d_j - n),\end{equation} where $C(n)$ is a positive constant depending only on $n$.

\begin{claim}\label{cldegree}
$d_1, .. , d_n$ (counting multiplicity) are independent of tangent cones. In other words, they depend only on $M$.
\end{claim}
\begin{proof}

Assume $a_1<a_2<\cdot\cdot\cdot<a_k$, where $a_j$ are the distinct values of $d_1,d_2,..., d_n$. Let $b_1,..., b_k$ be the multiplicities of $a_1, ..., a_k$.
By ~\cite[Proposition $4.1$]{[L4]}, for all $d>0$, $$dim(\mathcal{O}_d(M))=dim(\mathcal{O}_d(V)).$$ That is, as a function of $d$, $dim(\mathcal{O}_d(V))$ is independent of $V$.
Recall that $V$ is isomorphic to $\mathbb{C}^n$. 
We claim that $a_j, b_j$ are determined by $dim(\mathcal{O}_d(V))$. Observe that
$$a_1=\min\{d|dim(\mathcal{O}_d(V))>1\},  b_1=dim(\mathcal{O}_{a_1}(V)) - 1.$$ Let $\{1, h^1_1, .., h^1_{b_1}\}$ be a basis for $\mathcal{O}_{a_1}(V)$. Let $$\mathcal{O}^1_d(V)= \mathbb{C}[h^1_1, ..., h^1_{b_1}]\cap\mathcal{O}_d(M).$$ Then $dim(\mathcal{O}^1_d(V))$ is the number of solutions to $\sum\limits_{t=1}^{b_1} a_1x_{1t}\leq d$, where $x_{1t}$ are nonnegative integers. Hence $a_1, b_1, dim(\mathcal{O}^1_d(V))$ are independent of $V$. Inductively, for $j\geq 2$, assume $a_{j-1}, b_{j-1}, dim(\mathcal{O}^{j-1}_d(V))$ are independent of $V$. Then we have $$a_j = \min\{d|dim(\mathcal{O}_d(V))>dim(\mathcal{O}^{j-1}_d(V))\}, b_j=dim(\mathcal{O}_{a_j}(V)) - dim(\mathcal{O}^{j-1}_{a_j}(V)).$$ Let $\{1, h^j_1, .., h^j_{s_j}\}$ be a basis for $\mathcal{O}_{a_j}(V)$. Define $$\mathcal{O}^j_d(V)= \mathbb{C}[h^j_1, ...,h^j_{s_j}]\cap\mathcal{O}_d(M).$$ Then $dim(\mathcal{O}^{j}_d(V))$ is the number of solutions to $\sum\limits_{s=1}^j a_s (\sum\limits_{t=1}^{b_s} x_{st})\leq d$, where $x_{st}$ are nonnegative integers.
This depends only on $a_1, b_1,...,a_j, b_j, d$. Thus, $a_j, b_j, dim(\mathcal{O}^{j}_d(V))$ are independent of $V$.
This completes the proof.

\end{proof}
Putting (\ref{eqforintegral}) and Claim \ref{cldegree} together, we obtain the proof of Theorem \ref{thmniquestion}, if the universal cover of $M$ does not split.
If the universal cover of $M$ splits, by $\lim\limits_{r\to\infty}\sup r^2\dashint_{B(p, r)}S < \infty$ and ~\cite[Theorem $1.2$]{[NT2]}, we can find a psh function $u$ of logarithmic growth such that $dd^c u = Ric$.
We can lift $u$ to the universal cover $\tilde{M}$, thus $$dd^c u = Ric$$ on $\tilde{M}.$ Assume $\tilde{M}$ is biholomorphic and isometric to $M_1\times M_2\times\cdot\cdot\cdot \times M_k\times\mathbb{C}^l$, where each $M_j$ is nonflat and non-splitting. Pick $q_j\in M_j$. 
Then $v_j=u|_{(q_1, q_2, ...,  q_{j-1}, M_j, q_{j+1}, ..., q_k)}$ satisfies $dd^c v_j = Ric(M_j)$. As $v_j$ is of logarithmic growth, by \cite{[L3]}, $M_j$ is of Euclidean volume growth. 
Let $N= M_1\times M_2\times\cdot\cdot\cdot \times M_k$. 
\begin{claim}\label{clfixedpt}
Let $G$ be the holomorphic isometry group of $N$. Then $G$ is compact. Moreover, there exists a common fixed point of $G$.
\end{claim}
\begin{proof}
Pick $q'\in N$. If $G$ is not compact, there exists $g_i\in G$ such that $r_i = d(g_i(q'), q')\to \infty$. By passing to a subsequence, $(N_i, q'_i) = (N, q', \frac{r}{r_i})$ converges in the pointed Gromov-Hausdorff sense to a tangent cone. By Cheeger-Colding-Tian \cite{[CCT]}, such tangent cone splits off $\mathbb{C}$. By three circles lemma \cite{[L1]}, we find a linear growth holomorphic function on $N$. Then by ~\cite[Theorem $0.3$]{[NT1]}, $N$ itself splits isometrically with a $\mathbb{C}$ factor. This contradicts the assumption that $N$ has no flat factor.
Let $a_1<a_2<\cdot\cdot\cdot<a_k$, $b_1, ..., b_k$ be defined as in Claim \ref{cldegree} for $N$. According to the bottom of page $55$ in \cite{[L4]}, we can find polynomial growth holomorphic functions $f_{ij}$ ($1\leq i\leq k, 1\leq j\leq b_i$) so that $f_{ij}$ has degree $a_i$ (i.e., $a_i = \inf\{d|f_{ij}\in\mathcal{O}_d(N)\}$) and that $\{f_{ij}\}$ serves as a biholomorphism from $N$ to $\mathbb{C}^m$ ($m=dim(N))$. 
For any $g\in G$, as $g$ preserves the degree of polynomial growth holomorphic functions, for $1\leq s\leq b_1$, $$g(f_{1s}) = \sum\limits_{1\leq t\leq b_1} a_{st}f_{1t} + C_s.$$ We may regard $x_s = f_{1s}$ as variables. Set $X = (x_1, ..., x_{b_1})^t\in\mathbb{C}^{b_1}$. Then there exists a matrix $A_g$ and a vector $Y_g$ such that $gX = A_gX + Y_g$. In other words, $G$ acts as affine transformations on $\mathbb{C}^{b_1}$. As $G$ is compact, $G$ admits a Haar measure $dg$ with total volume $1$. set $X_0 = \int_G g(X)dg$. Then $g(X_0)=X_0$ for any $g\in G$. By shifting the coordinates of $\mathbb{C}^{b_1}$ (we shift $f_{1s}$ accordingly), we may assume $X_0 = 0$. Therefore, the common zeros of $f_{1s}$, say $N_1$, is invariant under $G$ action. Recall $f_{1s}$ forms a part of global coordinate on $N$ isomorphic to $\mathbb{C}^m$, then $N_1$ is identified with $\mathbb{C}^{m_1}$, where $m_1 = m - b_1$. Then we can consider the action of $G$ on polynomial growth holomorphic functions with degree $a_2$, restricted to $N_1$. Note $f_{1s}$ all vanish on $N_1$. As before, we obtain affine action $G$ acting on $\mathbb{C}^{b_2}$. Then we can find a fixed point of $G$ on $\mathbb{C}^{b_2}$, which defines $N_2\subset N_1$. By repetition, we obtain a fixed point of $G$ on $N$.

\end{proof}
Let $a\in N$ be a fixed point of $G$ ensured by Claim \ref{clfixedpt}. Let $q = (a, 0) \in N\times\mathbb{C}^l$. Let $G' = \pi_1(M)$. Then $G'$ acts as deck transformations on $N\times\mathbb{C}^l$. Define $$\Sigma = (a, \mathbb{C}^l)\subset N\times \mathbb{C}^l.$$ By projection, $G'$ also acts on $N$ holomorphically and isometrically. As $a$ is a fixed point, $G'(\Sigma) = \Sigma$. Hence, the action of $G'$ on $\Sigma$ is free and discrete. Then $M'=\Sigma/G'$ is a flat manifold.  Let $$L = \{x\in\Sigma|d(x, q)<d(x, g(q)), \forall g\in G', g\neq Id\}$$ be the fundamental domain of $G'$ on $\Sigma$.  Let $q_0\in M$ be the image of $q$ under the covering map $N\times\mathbb{C}^l\to M$. For large $r$, consider $B(q_0, r)\subset M$. Up to a zero measure set, $B(q_0, r)$ is identified with $$\Omega_r=\{(y, x)\in B(q, r)\subset N\times\mathbb{C}^l | d((y, x), q)<d((y, x), g(q)), \forall g\in G', g\neq Id\}.$$ 
Because $g(a) = a$ for all $g$, $$\Omega_r = \{(y, x)|x\in L, d(y, a)<\sqrt{r^2 - |x|^2}\},$$ where $|x|$ is the Euclidean distance to $0$ in $\mathbb{C}^l$.  Let $$L_r = L\cap \{x||x|<r\}.$$
When $r$ is fixed, for each $x\in L$, let $$\Sigma_x = \Omega_r\cap (N\times \{x\}).$$ Then $\Sigma_x$ can be identified with $B(a,\sqrt{r^2 - |x|^2})\subset N$. For any $\epsilon>0$, 

 \begin{equation}\label{forlimit}\begin{aligned}\frac{r^2}{vol(B(q_0, r))}\int_{B(q_0, r)}S &= \frac{r^2}{vol(B(q_0, r))}\int_{L_r}\int_{\Sigma_x}Sd\Sigma_xdx \\&= \frac{r^2}{vol(B(q_0, r))}\int_{L_r}(\int_{B(a, \sqrt{r^2-|x|^2})}S(y)dy)dx \\&= \frac{r^2}{vol(B(q_0, r))}\int_{L_r\backslash L_{(1-\epsilon)r}}(\int_{B(a, \sqrt{r^2-|x|^2})}S(y)dy)dx \\&+ \frac{r^2}{vol(B(q_0, r))}\int_{L_{(1-\epsilon)r}}(\int_{B(a, \sqrt{r^2-|x|^2})}S(y)dy)dx.\end{aligned}
   \end{equation}  Since $L$ is convex, $$\frac{vol(L_{(1-\epsilon)r})}{vol(L_r)}\geq (1-\epsilon)^l.$$ Then from the existence of $A=\lim\limits_{r\to\infty}r^2\dashint_{B(a, r)}S$, we find $$ \frac{r^2}{vol(B(q_0, r))}\int_{L_r\backslash L_{(1-\epsilon)r}}(\int_{B(a, \sqrt{r^2-|x|^2})}S(y)dy)dx\leq \frac{r^2\int_{B(a, r)}S(y)dy}{vol(B(q_0, r))}\int_{L_r\backslash L_{(1-\epsilon)r}}dx<\Psi(\epsilon).$$  In order to estimate $$\frac{r^2}{vol(B(q_0, r))}\int_{L_{(1-\epsilon)r}}(\int_{B(a, \sqrt{r^2-|x|^2})}S(y)dy)dx, $$ we can replace the integrand $\int_{B(a, \sqrt{r^2-|x|^2})}S(y)dy$ by $A(r^2-|x|^2)^{-1}vol(B(a, \sqrt{r^2-|x|^2}))$, up to $\Psi(\frac{1}{r\epsilon})$ error. This is approximately $vA(r^2-|x|^2)^{m-1}$, where $v$ is the asymptotic volume ratio for $N$. Therefore, in order to show that the limit of (\ref{forlimit}) exists, it suffices to show that for $t<1$, $$\lim\limits_{r\to\infty}\frac{vol(L_{tr})}{vol(L_{r})}$$ exists. Recall $M' = \Sigma/G'$ is a flat manifold. $L_r$ is just $B_{M'}(0, r)$, up to zero measure set. Then we are left to show that $\lim\limits_{r\to\infty}\frac{vol(B_{M'}(0, tr))}{vol(B_{M'}(0, r))})$ exists. By a theorem of Cheeger-Gromoll \cite{[CG]}, such $M'$ has a soul $K$, which is a compact flat manifold of fundamental group $G'$.  The normal bundle to the soul is a flat orthogonal vector bundle over $K$, which is given by a representation $G'\rightarrow O(s)$.  The normal exponential map is an isometric diffeomorphism. In particular, up to a codimension $1$ set, $M'$ is isometric to $K\times\mathbb{R}^s$. Therefore,$$\lim\limits_{r\to\infty}\frac{vol(B_{M'}(0, tr))}{vol(B_{M'}(0, r))}=t^s.$$    
Therefore, we can find $C(m, s)$ depending only on $m$ and $s$ so that $$\lim\limits_{r\to\infty}\frac{r^2}{vol(B(q_0, r))}\int_{B(q_0, r)}S = C(m, s)A.$$ By considering a special case, e.g., $M = N^m\times \mathbb{C}^s$, we find $$C(m, s) = \frac{2(m+s)-1}{2m-1}.$$
This completes the proof of Theorem \ref{thmniquestion}.

\medskip
\medskip
Through the proof, we also obtained the following:
\begin{cor}
Let $(M^n, p)$ be a complete noncompact K\"ahler manifold with nonnegative bisectional curvature. If $\lim\limits_{r\to\infty}r^2\dashint_{B(p, r)}S$ is finite, then $\pi_1(M)$ is isomorphic to $\pi_1(M')$, where $M'$ is a complete flat K\"ahler manifold of dimension at most $n$.
\end{cor}
\medskip

Proof of Theorem \ref{thmsharpestimates}:

Let $d_1, d_2, ..., d_n$ be the degrees appeared in the proof of Theorem \ref{thmniquestion}.
 Consider a special K\"ahler manifold $M'$ given by $S_1\times S_2\times\cdot\cdot\cdot\times S_n$, where each $S_j$ is a complete surface with nonnegative curvature and quadratic volume growth, isomorphic to $\mathbb{C}$. We also require that the degree of $z$ on $S_j$ is $d_j$. The tangent cone of $M'$ and the tangent cone of $M$ have the same reeb vector field, up to an isomorphism. The volume of unit ball on tangent cone of $M'$ is computed as $\frac{\omega_{2n}}{d_1\cdot\cdot\cdot d_n}$. Such metric is singular. However, as we see in the proof of Theorem \ref{thmniquestion}, it can be approximated by smooth K\"ahler cone metrics with the same reeb vector field. By Chern-Levine-Nirenberg and the convergence of K\"ahler potentials, we find that the volume of unit ball of the smooth cone metric converges to $\frac{\omega_{2n}}{d_1\cdot\cdot\cdot d_n}$. By ~\cite[Section $3.1$]{[MSY]}, such volume depends only on the reeb vector field. So volume of unit ball for smooth K\"ahler cone metric is $\frac{\omega_{2n}}{d_1\cdot\cdot\cdot d_n}$. We also approximate a tangent cone of $M$ by smooth K\"ahler cone metrics. Then the volume of unit ball in tangent cone of $M$ is also $\frac{\omega_{2n}}{d_1\cdot\cdot\cdot d_n}$.

\medskip

Inequalities $(1)$, $(2)$ follow by the fact that $d_i\geq 1$ and the arithmetic-geometric inequality. 
Equality $(3)$ follows from straightforward dimension counting, using the fact that it corresponds to the number of solutions to $\sum d_ja_j\leq d$, where $a_j\geq 0$ are nonnegative integers.
The rigidity part also follows. In one extremal case, $d_1=d_2=\cdot\cdot\cdot=d_{n-1}=1$. By ~\cite[Theorem $0.3$]{[NT1]}, we find that $M$ splits off $\mathbb{C}^{n-1}$. 
In another case, $d_1=d_2=\cdot\cdot\cdot=d_n$. This is the case when $M$ is unitary symmetric \cite{[WZ]}. However, the metric is not necessarily unique: if the bisectional curvature is strictly positive, then we can do local perturbation, which does not affect the holomorphic spectrum. This completes the proof of Theorem \ref{thmsharpestimates}.

\medskip
\medskip

Proof of Corollary \ref{niquantitative}:

By the last part in the proof of Theorem \ref{thmniquestion}, we find that the universal cover $\tilde{M}$ has Euclidean volume growth. Assume $$\lim\limits_{r\to\infty}r^2\dashint_{B(p, r)}S<\delta<<1,$$ Then $$\lim\limits_{r\to\infty}r^2\dashint_{B(\tilde{p}, r)}S<C(n)\delta<<1.$$
Then by Theorem \ref{thmsharpestimates}, $$\lim\limits_{r\to\infty}\frac{vol(B(\tilde{p}, r))}{\omega_{2n}r^{2n}}\geq 1 - \Psi(\delta).$$ The opposite direction holds similarly. This completes the proof of Corollary \ref{niquantitative}.

\medskip
\medskip

Proof of Theorem \ref{bigvolume}:
Recall by ~\cite[Proposition $2.7$]{[LS1]}, the manifold is biholomorphic to $\mathbb{C}^n$. By adapting the argument in ~\cite[Proposition $6.1$]{[Lo]} and ~\cite[Theorem $1.5$]{[LS1]}, we can show that the tangent cone is also biholomorphic to $\mathbb{C}^n$. We can still apply the argument as in the proof of Theorem \ref{thmniquestion}. The difference is that, we do not have three circles lemma. So apriori, $d_1,...,d_n$ could vary for different tangent cones.
However, this does not affect the argument for a single tangent cone. If $\lim\limits_{r\to\infty}\inf r^2\dashint_{B(p, r)}S=0$, then for some tangent cone, the coordinate functions are of linear growth and the volume of unit ball is the same as the Euclidean case. Thus by Colding's volume continuity \cite{[Coldingvolume]}, $M$ is flat.
This completes the proof of Theorem \ref{bigvolume}.

\section {The case of nonnegative Ricci curvature}

Proof of Theorem \ref{ricflatmain1}:

Let us first assume $M$ satisfies condition $B$.
\begin{lemma}\label{highpowerbd}
Let $M^n$ satisfy condition $B$. Then there exists $C_1>0$ such that for all $1\leq k\leq n$, for all $r>0$, 
$$r^{2k-2n}\int_{B(p, r)}Ric^k\wedge \omega^{n-k}<C_1.$$ Furthermore, if we consider a tangent cone $(V, o)$ and a sequence of rescaled manifolds $(M_i, p_i) = (M, p, \frac{r}{r_i})\to (V, o)$ in the pointed Gromov-Hausdorff sense, then for $k<n$, $$\lim\limits_{i\to\infty}r_i^{2k-2n}\int_{B(p, r_i)}Ric^k\wedge \omega^{n-k}=\int_{B_V(o, 1)}Ric_V^k\wedge \omega_V^{n-k},$$ where $Ric_V$ is understood in the current sense.

\end{lemma}
\begin{proof}
Assume by contradiction that there exists a sequence $r_i\to\infty$ such that \begin{equation}\label{goingtoinfinity}r_i^{2k-2n}\int_{B(p, r_i)}Ric^k\wedge\omega^{n-k}\to\infty.\end{equation}
Set $(M_i, p_i) = (M, p, \frac{r}{r_i})$. By passing to a subsequence, we may assume $(M_i, p_i)\to (V, o)$. According to \cite{[DS]} and page $1161$ of \cite{[Sz]}, there exists $u_i$ on $B(p_i, 10)$ such that \begin{equation}\label{equationforui}dd^c u_i = \omega_i,|u_i - d^2(p_i, \cdot)|<\Psi(\frac{1}{i}),|u_i|+|\nabla u_i|\leq C.\end{equation} We can find holomorphic functions $f_1, ..., f_n$ on $V$ so that $df_1\wedge\cdot\cdot\cdot\wedge df_n$ is not identically zero. Note $df_1\wedge\cdot\cdot\cdot\wedge df_n$ could have divisor. By solving $\overline\partial$ equation, $f_j$ can be lifted from tangent cone to rescaled manifolds $B(p_i, 10)$, say $f_{ij}\to f_j$. Let $$s_i = df_{i1}\wedge\cdot\cdot\cdot\wedge df_{in}.$$
Let $\chi$ be a nonnegative smooth function on $\mathbb{R}^+$ so that $\chi = 1$ on $[0, 1]$, with compact support in $[0, 1.5)$. By Poincar\'{e}-Lelong, we have  \[ \begin{aligned}2\pi\int \chi(u_i) Ric_i^k\wedge\omega_i^{n-k}&\leq \int\chi(u_i) dd^c\log |s_i|^2\wedge Ric_i^{k-1}\wedge\omega_i^{n-k}\\& = \int \log |s_i|^2 dd^c\chi(u_i)\wedge Ric_i^{k-1}\wedge\omega_i^{n-k}\\& \leq C\int_{B(p_i, 2)\backslash B(p_i, 1)} |\log |s_i|^2| \omega_i^{n},\end{aligned}\] using $Ric\leq C\omega$ on $B(p_i, 2)\backslash B(p_i, 1)$.
Then it suffices to find uniform bound of $\int \log |s_i|^2\omega_i^n$ on an annulus. Absolute value is not needed, as $\log |s_i|^2$ is uniformly bounded from above. Without loss of generality, we may assume $\log |s_i|^2 \leq 0$. Since there may be singularity on cross section of $V$, we consider projection. According to the local parametrization theorem in \cite{[De2]} (Theorem $4.19$ on page $95$), for any point $q\in\partial B(o, 1)$, we can locally properly project $B(q, \delta)$ to a domain in $\mathbb{C}^n$. Say the projection is given by $P$. Then under the Gromov-Hausdorff convergence, we can also project $B(q_i, \delta)$, where $q_i\to q$, say the projection is $P_i$ and $P_i\to P$. Recall $\omega_i = dd^c u_i$. Assume $P_i(q_i) = 0, P(q) = 0$ and there exists $\delta_1>0$ with $P_i^{-1}(B(0, \delta_1))\subset\subset B(q_i, 0.5\delta)$. For any $y\in B(0, \delta_1)$, define $$h_i(y) = \sum\limits_{x\in P_i^{-1}(y)}\log |s_i(x)|^2\leq 0, v_i(y) = \sum\limits_{x\in P_i^{-1}(y)}u_i(x).$$ Then $h_i, v_i$ are psh and $v_i$ is uniformly bounded. We may assume $P_i(B(q_i, \delta_2))\subset B(0,0.5\delta_1)$. 
Then $$\int_{B(q_i, \delta_2)} -\log |s_i|^2\omega_i^n\leq -\int_{B(0, 0.5\delta_1)} h_i(dd^cv_i)^n\leq C.$$ In the last step, we used Proposition $3.11$ on page $158$ of \cite{[De2]} and that $h_i$ has a convergent subsequence as psh function, hence the $L^1$ norm (w.r.t Euclidean volume) is uniformly bounded. Then by a standard covering argument, we find a contradiction to (\ref{goingtoinfinity}).

Now we come to the convergence of integral for $k<n$. As we proved that $$\int_{B(p, r)}Ric^k\wedge\omega^{n-k}\leq Cr^{2n-2k},$$ it sufficies to prove the convergence on an annulus (the integral on the small ball is controlled by $Cr^{2n-2k}$). Away from singular set, as we have $C^{1, \alpha}$ convergence \cite{[A1]}, the potentials of $Ric_i$ and $\omega_i$ converge. Then the proof follows from Chern-Levine-Nirenberg. Near the singular set, as $Ric\leq C\omega$, then the integral is bounded by a constant times the volume of singular set. So the contribution is negligible. This completes the proof of Lemma \ref{highpowerbd}.
\end{proof}

Consider one tangent cone $(W, o')$ which is Ricci flat. Let us consider a sequence $r_i\to\infty$, such that $(M_i, p_i) = (M, p, \frac{r}{r_i})$ converges in the pointed-Gromov-Hausdorff sense to $(W, o')$. Let $s$ be the parallel pluri-anticanonical form on $W$ (dual of the pluri-canonical form), say $s\in H^0(W_{reg}, K^{-m})$. Let $\delta_0>0$ be a small number. Let $\Sigma'$ be the metric singular set of $W$. We can lift $\Sigma'$ to $B(p_i, 9)$ by Gromov-Hausdorff approximation. Say the corresponding set is $\Sigma_i'$. By pulling $s$ back to $B(p_i, 9)$, we obtain smooth sections $s''_i\to s$ uniformly on each compact set of $B(p_i, 9)\backslash \Sigma_i'$. Set $S_i = s_i''\eta_{i1}\eta_{i2}$, where $\eta_{i1}$ vanishes on $B(p_i, 2\delta_0)$ and equals to $1$ outside $B(p_i, 3\delta_0)$, $\eta_{i2}$ vanishes on $B(\Sigma_i', \delta_0)$, and equals to $1$ outside $B(\Sigma_i', 2\delta_0)$. Let $u_i$ be defined as in (\ref{equationforui}). Let $\Omega_i$ be the connected component of  $u_i^{-1}([0, 5))$ containing $B(p_i, 1)$. Then $$ B(p_i, 3)\supseteq\Omega_i\supseteq B(p_i, 2).$$ Since $n\geq 2$ and $\Sigma'$ has real codimension at least $4$ by Cheeger-Naber \cite{[CN]},  we obtain $S_i$ satisfies $$\int_{\Omega_i} |\overline\partial S_i|^2e^{-2u_i}\leq\Psi(\delta_0, \frac{1}{i}).$$  Note here we crucially used that $s$ is parallel, hence the norm is constant. Let $(g')^i_m$ be the metric on $K_{M_i}^{-m}$ induced by the K\"ahler metric. As $Ric\geq 0$, $(g')^i_m$ has nonnegative curvature. Let the new metric be $g^i_m=e^{-u_i}(g')^i_m$. 
 We solve $\overline\partial \hat{s}_i= \overline\partial S_i$, by using $(g')^i_m$. H\"ormander $L^2$ estimate of $\overline\partial$ gives that 
$$\int_{\Omega_i}|\hat{s}_i|^2e^{-2u_i}\leq 10\int_{\Omega_i} |\overline\partial S_i|^2e^{-2u_i}\leq\Psi(\delta_0, \frac{1}{i}).$$ 
Let $s_i = S_i-\hat{s}_i$. Then $s_i\in H^0(\Omega_i, K^{-m})$ and \begin{equation}\label{l2bound}\int_{\Omega_i} |s_i|^2 < C.\end{equation} Recall the scalar curvature has quadratic decay. Hence, on $B(p_i, 2)\backslash B(p_i, 1)$, we have:  $$\frac{1}{2\pi} dd^c \log |s_i|^2 \geq -mRic_i - \omega_i \geq -C\omega_i.$$ Therefore $$\Delta \log (e^{Cu_i}|s_i|^2)\geq 0.$$ This implies that $$\Delta (e^{Cu_i}|s_i|)\geq 0).$$ 
Then (\ref{l2bound}) and mean inequality imply the pointwise upper bound for $|s_i|$ on $B(p_i, 2)\backslash B(p_i, 1)$. Let $\delta_0\to 0$ slowly as $i\to\infty$, we obtain that $s_i$ converges uniformly on each compact set of the regular part to $s$. We claim $s_i$ has no zero point on $B(p_i, 2)\backslash B(p_i, 1)$, for large $i$. Assume not. It is clear the zero set is a divisor. 
One can apply the Poincar\'{e}-Lelong equation and projection as in Lemma \ref{highpowerbd}, to conclude that the projection of the divisor to $\mathbb{C}^n$ has uniform volume (w.r.t Euclidean measure) upper bound. Hence, there exists a subsequence so that the divisor converges to a divisor (as current). However, the projection of the singular set has codimension $4$, this means there exists a point on the limit divisor away from the metric singular set. Hence for large $i$, there exists a point on divisor and uniformly away from singular set. This contradicts that $s_i\to s$ on regular part. Once such section has no zero point, we can find the dual of $s_i$ in pluri-canonical bundle, say $\mu_i$, which is holomorphic. Then by multiplying the metric by $e^{Cu_i}$, we may assume $\log |\mu_i|^2$ is psh. We claim that $\log |\mu_i|^2$ is uniformly bounded from above. Assume not, there exists a sequence of points $B(p_i, 2)\backslash B(p_i, 1)\ni q_i\to q$ so that $\log |\mu_i(q_i)|^2\to \infty$. Again we can apply local parametrization theorem to project via $P_i$ to a ball on $\mathbb{C}^n$, and define $\nu_i(y) =\sum\limits_{x\in P_i^{-1}(y)}\log |\mu_i(x)|^2$. Since $|s_i|^2$ has uniform upper bound, $\log |\mu_i|^2$ has uniform lower bound. Hence $\nu_i(P_i(q_i))\to \infty$. Note the metric singular set $\Sigma'$ is a proper analytic set. Now pick a generic complex line $l$ through $P(q)$ ($P$ is the limit of $P_i$), then locally $l$ intersects $P(\Sigma')$ at only finitely many points. In particular, a small (fixed) circle on $l$ around $P(q)$ has no intersection with $P(\Sigma')$. Now we can shift the complex line $l$ slightly so that it passes through $P_i(q_i)$. By maximum principle for psh functions, the supremum of $\nu_i$ on that circle is at least $\nu_i(P_i(q_i))$. But that small circle has fixed distance to $P(\Sigma')$. This is a contradiction.
Thus there exists a constant $C$ (independent of $i$) such that on $B(p_i, 2)\backslash B(p_i, 1)$, $$|\Delta \log |s_i|^2|\leq C, |\log |s_i|^2|<C.$$ We claim that $|\nabla \log |s_i|^2|$ is uniformly bounded. Let $G_R(x, y)$ be Green's function with boundary value $0$ on a geodesic ball in $B(p_i, 2)\backslash B(p_i, 1)$. Then $$\log |s_i(x)|^2 + \int G_R(x, y)\Delta \log |s_i(y)|^2dy$$ is harmonic and bounded in the interior. Hence the gradient is bounded. We can take gradient for each term. By the estimate for $|\nabla G_R(x, y)|$, we find that $|\nabla \log |s_i(x)|^2|$ is uniformly bounded. Arzela-Ascoli theorem implies $|s_i|$ on $B(p_i, 2)\backslash B(p_i, 1)$ converges uniformly to constant $s$ on $B(o', 2)\backslash B(o', 1)$. Since $\omega$ is $dd^c$-exact, $M$ has no compact subvariety of positive dimension. Hence, $|s_i|^2$ is nonvanishing anywhere. Then \begin{equation}\label{ricequality}Ric_{M_i} = dd^c\log |s_i|^{-\frac{2}{m}}.\end{equation}

Certainly such argument works, if $B(p_i, 10)$ is Gromov-Hausdorff close to $B(o, 10)$. Therefore, if a tangent cone $V$ sufficiently close to $W$, we can consider a sequence $r'_i\to\infty$ so that $(M'_i, p'_i) = (M, \frac{r}{r'_i}, p)$ converges in the pointed-Gromov-Hausdorff sense to $(V, o)$. We may assume that $r'_i\geq 10r_i$ for all $i$. Let $\Omega'_i, u'_i$ be defined similarly as $\Omega_i, u_i$. Then we find $s'_i$ converging to $s_V\in H^0(V_{reg}, K^{-m})$. Also $s_V$ is nowhere vanishing. As $r'_i\geq 10r_i$, by naturally identifying $M'_i$ and $M_i$ with $M$, we find $\overline{\Omega_i}\subset B(p, r'_i)$. 
Then we have $$\frac{1}{2\pi}dd^c\log |s'_i|^{-\frac{2}{m}} = Ric_{M'_i}=Ric_M.$$  Therefore, on $B(p, r_i)$, \begin{equation}\label{ddcsame}dd^c \log |s'_i|^2 = dd^c\log |s_i|^2.\end{equation}
Then by Lemma \ref{highpowerbd} (take $k=n$), we find that \begin{equation}\label{limitsame}\int_{B(p'_i, 2)}(dd^c\log |s'_i|^{-\frac{2}{m}})^n\to \int_{M} Ric^n, \int_{B(p_i, 2)}(dd^c\log |s_i|^{-\frac{2}{m}})^n\to \int_{M} Ric^n.\end{equation}

Let $\chi_i=h(u'_i)$, where $h$ is a smooth function on $[0, 2)$, $h\equiv 1$ on $[0, 1]$, and $h$ has compact support. Then \begin{equation}\label{l2conv3}\int \chi_i Ric_i^n=\int_{B(p'_i, 2)\backslash B(p'_i, 1)} \log |s'_i|^{-\frac{2}{m}}dd^c\chi_i \wedge (dd^c\log |s'_i|^{-\frac{2}{m}})^{n-1}.\end{equation} Note on $B(p'_i, 2)\backslash B(p'_i, 1)$, $|dd^c\chi_i|\leq C\omega'_i$, $Ric_{M'_i}\leq C\omega'_i$, also $|\log |s_i|^2|$ is uniformly bounded. Assume $\chi_i \to \chi$. Recall $s'_i\to s_V\in H^0(V_{reg}, K^{-m})$. We assert that \begin{equation}\label{frommtov}\begin{aligned}\lim\limits_{i\to\infty}\int \log |s'_i|^{-\frac{2}{m}}dd^c\chi_i \wedge (dd^c\log |s'_i|^{-\frac{2}{m}})^{n-1}&= \int \log |s_V|^{-\frac{2}{m}}dd^c \chi \wedge (dd^c\log |s_V|^{-\frac{2}{m}})^{n-1}\\&=\int_V\chi (dd^c\log |s_V|^{-\frac{2}{m}})^n.\end{aligned}\end{equation}
Certainly the potential convergence holds on each compact set of regular parts, hence the integral converges by Chern-Levine-Nirenberg. In a small neighborhood of singular part, note the integrand is bounded by $C\omega_{M'_i}^n$, hence the contribution is negligible. This proves (\ref{frommtov}). We can also apply (\ref{frommtov}) to the sequence $s_i\to s$. Putting (\ref{limitsame}), (\ref{l2conv3}) and (\ref{frommtov}) together,
 we find $$\int_{V}\chi (dd^c\log |s_V|^{-\frac{2}{m}})^n = 0.$$ Let $\Omega\in H^0(V_{reg}, K^m)$ be the dual of $s_V$, which is nowhere vanishing. Then $\log |\Omega|^2$ is psh and \begin{equation}\label{integralvanish}\int_{B(o, 2)}\chi(dd^c\log|\Omega|^2)^n = 0.\end{equation}

As in \cite{[MSY]} and Lemma $6.1$ in \cite{[CS]}, by multiplying by a nowhere vanishing holomorphic function, we may assume $\Omega$ is homogeneous with respect to $r\frac{\partial}{\partial r}$. We remark that $\log |\Omega|$ can not be strictly decreasing along $r\frac{\partial}{\partial r}$. Otherwise, for some $c>0$, we may assume $\log |\Omega| + c\log r$ is homogeneous and the norm is constant along $r\frac{\partial}{\partial r}$. Such function is psh, and by maximum principle, it is constant. Hence, the psh function $\log |\Omega| = c_1-c\log r$. This is not possible when $n\geq 2$. 

Let $u = \log |\Omega|^{\frac{2}{m}}$, $v = c\log r^2$, where $c$ is chosen so that they have the same degree along $r\frac{\partial}{\partial r}$. Let $A= B(o, 2)\backslash B(o, 0.1)$.  Then on $A_{reg}$, \begin{equation}\label{uvbounds}0\leq dd^c u \leq C\omega, 0\leq dd^c v\leq C\omega, |u|+|v|\leq C, |du|+|dv|\leq C.\end{equation}

\begin{lemma}\label{sameintegral}
\begin{equation}\label{integralsamesingular}\int_{A_{reg}}u dd^c\chi\wedge (dd^c u)^{n-1} = \int_{A_{reg}}v dd^c\chi\wedge (dd^c v)^{n-1}.\end{equation}
\end{lemma}

\begin{proof}
 Let $\eta$ (depending on $\epsilon$) be a smooth cut-off function which is supported on $A\backslash B_\epsilon(\Sigma)$, where $\Sigma$ is the metric singular set, $0\leq \eta\leq 1$. Also assume $\eta\equiv 1$ on $A\backslash B_{2\epsilon}(\Sigma)$, $|d\eta|\leq C/\epsilon$. According to (\ref{uvbounds}), $$\lim\limits_{\epsilon\to 0}\int_{A}\eta u dd^c\chi\wedge (dd^c u)^{n-1}= \int_{A_{reg}}u dd^c\chi\wedge (dd^c u)^{n-1}.$$ We have 
 \begin{equation}\label{uminusvterm}\begin{aligned}&\int_{A}\eta u dd^c\chi\wedge (dd^c u)^{n-1} - \int_{A}\eta v dd^c\chi\wedge (dd^c u)^{n-1}\\& = \int_{A}\eta d(v-u)\wedge d^c\chi\wedge (dd^c u)^{n-1}+ \int_{A} (v-u)d\eta\wedge d^c\chi\wedge (dd^c u)^{n-1}.\end{aligned}\end{equation}

For any $q\in A_{reg}$, choose a holomorphic chart containing $q$ so that $\frac{\partial}{\partial z_1} = r\frac{\partial}{\partial r} - \sqrt{-1}Jr\frac{\partial}{\partial r}$. Now in such coordinate, if we set $x_1=Re z_1$, then $$u = kx_1 + f_1(z_2, .., z_n), v = kx_1 + f_2(z_2, .., z_n), v-u=g(z_2, .., z_n).$$
Now assume in addition that $u, v$ are smooth. Under the local coordinate above, for all $s$, we have $u_{1\overline{s}} = v_{s\overline{1}} = 0$, also $(v-u)_1=(v-u)_{\overline{1}} = 0$, hence $$d(v-u)\wedge d^c\chi\wedge (dd^c u)^{k}\wedge (dd^c v)^{n-k-1}\equiv 0.$$ The nonsmoothness of $u, v$ can be handled by equivariant local smooth approximation on the chart above. 
Hence the first term in the second line of (\ref{uminusvterm}) vanishes identically. The second term approaches zero as $\epsilon\to 0$, since we have (\ref{uvbounds}) and that the singular set has real codimension $4$.
We claim that $$\int_{A_{reg}}v dd^c\chi\wedge (dd^c u)^k\wedge (dd^c v)^{n-1-k} = \int_{A_{reg}}u dd^c\chi\wedge (dd^c u)^{k-1}\wedge (dd^c v)^{n-k}.$$ This directly follow from integration by parts, if $A$ is smooth. In the general case, \[ \begin{aligned}&\int_{A_{reg}}v dd^c\chi\wedge (dd^c u)^k\wedge (dd^c v)^{n-1-k} \\&= \lim\limits_{\epsilon\to 0}\int_{A_{reg}}\eta v dd^c\chi\wedge (dd^c u)^k\wedge (dd^c v)^{n-1-k} 
\\&= -\lim\limits_{\epsilon\to 0}\int_{A_{reg}}d(\eta v)\wedge dd^c\chi\wedge d^cu\wedge (dd^c u)^{k-1}\wedge (dd^c v)^{n-1-k}\\&=-\lim\limits_{\epsilon\to 0}\int_{A_{reg}}(\eta dv + vd\eta)\wedge dd^c\chi\wedge d^cu\wedge (dd^c u)^{k-1}\wedge (dd^c v)^{n-1-k}
\\&= -\lim\limits_{\epsilon\to 0}\int_{A_{reg}}\eta dv\wedge dd^c\chi\wedge d^cu\wedge (dd^c u)^{k-1}\wedge (dd^c v)^{n-1-k}\\&=-\lim\limits_{\epsilon\to 0}\int_{A_{reg}}ud^c(\eta dv)\wedge dd^c\chi\wedge (dd^c u)^{k-1}\wedge (dd^c v)^{n-1-k}\\&=\int_{A_{reg}}u dd^c\chi\wedge (dd^c u)^{k-1}\wedge (dd^c v)^{n-k}.\end{aligned}\] Therefore, we have \[ \begin{aligned}&\int_{A_{reg}}u dd^c\chi\wedge (dd^c u)^{n-1} \\&= \int_{A_{reg}}v dd^c\chi\wedge (dd^c u)^{n-1} \\&= \int_{A_{reg}}u dd^c\chi\wedge (dd^c u)^{n-2}\wedge dd^c v \\&= \int_{A_{reg}}v dd^c\chi\wedge (dd^c u)^{n-2}\wedge dd^c v\\&=\cdot\cdot\cdot=\int_{A_{reg}}v dd^c\chi\wedge (dd^c v)^{n-1}.\end{aligned}\] 

\end{proof}

\begin{lemma}
$\log |\Omega|$ is constant along $r\frac{\partial}{\partial r}$.
\end{lemma}
\begin{proof}
Assume $\log |\Omega|$ is not constant along $r\frac{\partial}{\partial r}$. Choose $c>0$ such that $\log|\Omega|^2- c\log r^2$ is constant along $r\frac{\partial}{\partial r}$. By Lemma \ref{sameintegral} and integration by parts, we find $$\int_{B(o, 1)}(dd^c\log|\Omega|^2)^n =c^n\int_{B(o, 1)}(dd^c\log r^2)^n.$$ 
According to formula $5.5$ (page 158) of \cite{[De2]}, 
$$\int_{B(o, 1)}(dd^c\log r^2)^n = \int_{B(o, 1)}(dd^c r^2)^n = vol(B(o, 1))>0.$$
Therefore, $$\int_{B(o, 1)}(dd^c\log|\Omega|^2)^n = c^nvol(B(o, 1))>0.$$
This contradicts (\ref{integralvanish}).

\end{proof}

 As $\log|\Omega|$ is psh, by maximum principle, $|\Omega|$ is constant. This means $V$ is Ricci flat in the current sense. By the regularity theory of complex Monge-Amp\`ere equations, we see that $V$ is Ricci flat in the differential geometric sense, and the metric is $C^\infty$ on regular part. Also $\Omega$ is a parallel pluri-canonical section: by parallel transport, we obtain a pluri-canonical form $\gamma$ locally, then $\frac{\gamma}{\Omega}$ is holomorphic with constant norm, hence $\frac{\gamma}{\Omega}$ is constant. This proves the openness of Ricci flat tangent cones with a parallel pluri-canonical section. Closedness follows from convergence theory by Anderson \cite{[A1]} and Cheeger-Naber \cite{[CN]}. By connectness of tangent cones, all tangent cones are Ricci flat.

Ricci flatness of all tangent cones imply that the holomorphic spectrum is rigid \cite{[MSY]}\cite{[DS]}. Then by \cite{[DS]}, $M$ is identified with an affine algebraic variety $X$ in $\mathbb{C}^N$ (we only need that the holomorphic spectrums are the same). Let $(z_1, ..., z_N)$ be the holomorphic coordinate on $\mathbb{C}^N$.

 Let $\hat{X}$ be the compactification of $X\subset \mathbb{C}^N\subset\mathbb{P}^N$. For $N$ large, we can find a sequence of automorphism $\tau_i$ of $\mathbb{P}^N$ and a sequence $r_i\to \infty$ such that if we define $(M_i, p_i) = (M, p, \frac{r}{r_i})$, then $\tau_i((M_i, p_i))$ converges in both intrinsic (the pointed-Gromov-Hausdorff sense) and extrinsic (current) sense to the tangent cone $(V, o)$. Also, there exists $c>0$ such that $$B_{V}(o, c)\subset B_{\mathbb{C}^N}(0, 1)\cap V\subset B_{\mathbb{C}^N}(0, 2)\cap V \subset B_{V}(0, \frac{1}{c}).$$  Therefore, for large $i$, 
$$B_{M_i}(p_i, 0.5c)\subset B_{\mathbb{C}^N}(0, 1)\cap \tau_i(M_i)\subset B_{\mathbb{C}^N}(0, 2)\cap \tau_i(M_i) \subset B_{M_i}(p_i, \frac{2}{c}).$$ Let us fix such $c$. Let the Fubini-Study metric on $\mathbb{P}^N$ be given by $$\omega_{FS}=dd^c\log (1+|z_1|^2+\cdot\cdot\cdot+|z_N|^2).$$
Assume $M$ is not Ricci flat. Then there exists $q \in M$ and $e\in T_qM$ such that $Ric(e, e)>0$. Consider an algebraic curve $L\subset X$, through $q$ with tangent proportional to $e$. Let $\hat{L}$ be the compactified curve. Then it has finite degree. Under the sequence of automorphism of $\mathbb{P}^N$, the degree is fixed. Hence, the volume w.r.t $\omega_{FS}$ is uniformly bounded. 

Since we have a nowhere zero section of $K_V^m$ with constant norm $1$, we can lift such the pluri-anticanonical dual section by $L^2$ method.
Say the section is $s_i$. Then $s_i$ is nowhere vanishing. Also, on the annulus $B(p_i,\frac{2}{c})\backslash B(p_i,0.5c)$, $|s_i|^2$ is converging uniformly to $1$. Thus by scale invariance of $Ric$, on $L_i=\tau_i(L)$, $$\frac{1}{m}dd^c (-\log |s_i|^2) = Ric_i=Ric.$$  
Let $h$ be a smooth function on $\mathbb{R}^+$ satisfying $h=1$ on $[0, 1]$, $h\geq 0$, $h$ has compact support on $[0, 1.5)$.  Let $\chi = h(\sqrt{|z_1|^2+|z_2|^2+\cdot\cdot\cdot+|z_N|^2})$. 
Now $$\int_{L_i} Ric\chi = \int_{L_i} \frac{1}{m}dd^c (-\log |s_i|^2)\chi = \int_{L_i\cap (B(p_i, \frac{2}{c})\backslash B(p_i, 0.5c))} \frac{1}{m}(-\log |s_i|^2)dd^c \chi.$$ Recall that $L_i$ has uniform area bound with respect to $\omega_{FS}$. As $\chi$ is supported on $B_{\mathbb{C}^N}(0, 1.5)$, $$|dd^c\chi|_{\omega_{FS}}<C.$$ Also $\log |s_i|^2$ converges uniformly to $0$ on $B(p_i, \frac{2}{c})\backslash B(p_i, 0.5c)$. Putting these together, we find \begin{equation}\label{ricintegralvanish}\lim\limits_{i\to\infty}\int_{L_i} Ric\chi = 0.\end{equation}

On the other hand, \begin{equation}\label{integralpositive}\int_{L_i} Ric\chi\geq \int_{L\cap B(q, 1)} Ric>\epsilon_0>0\end{equation} for some fixed $\epsilon_0$, since $Ric$ is positive at the tangent of $L$ at $q$. Then (\ref{integralpositive}) contradicts (\ref{ricintegralvanish}).
The uniqueness of tangent cone follows from \cite{[DS]}. This completes the proof of Theorem \ref{ricflatmain1} when $M$ satisfies condition $B$.

\medskip

If $M$ satisfies condition $A$, the argument is pretty much the same. 
Let us point out the minor differences. Note for the Ricci flat tangent cone, the cross section is smooth with $Ric = 2n-2$. Then by Bonnet-Meyers theorem, the fundamental group is finite. Hence the parallel pluri-canonical section exists automatically. Let $r_i$ be a sequence such that $r_i\to\infty$. Set $(M_i, p_i) = (M, p, \frac{r}{r_i})$. By passing to a subsequence, we may assume $M_i\to V$. By \cite{[L2]}, $B(p_i, 1)$ is $1$-convex and $F_i=(f_{i1}, .., f_{iN})$ defines a holomorphic map from $B(p_i, 1)\to \mathbb{C}^N$. $F_i$ is an isomorphism except finitely many compact subvarieties (let us call it $E$). Moreover, $f_{ij} = 0$ at $p_i$, $f_{ij}\to f_j$ on $V$ locally uniformly. $(f_1, ..., f_N)$ is a holomorphic embedding from $B_V(o, 1)\to \mathbb{C}^N$. 
As $\omega$ need not be $dd^c$-exact, one should replace (\ref{equationforui}) by a psh function $u_i$ on $B(p_i, 10)$ such that on $B(p_i, 10)\backslash B(p_i, \delta_0)$, $$dd^c u_i \geq 0.5\omega_i.$$ Also on $B(p_i, 10)$, $$|u_i - d^2(p_i, \cdot)|<\Psi(\frac{1}{i}), |u_i|+|\nabla u_i|\leq C.$$ These can be ensured by ~\cite[Proposition $2.5$]{[L2]}. Note that (\ref{ricequality}) holds on $B(p_i, 2)\backslash E$.
However, (\ref{ddcsame}) still holds. The reason is that $\frac{s'_i}{s_i}$ is holomorphic away from $E$, after contraction of $E$, $\frac{s'_i}{s}$ can be extended as a holomorphic function which is nowhere vanishing (use $n\geq 2$).
Correspondingly, (\ref{limitsame}) should be replaced by
\begin{equation}
\begin{aligned}\int_{B(p'_i, 2)}(dd^c\log |s'_i|^{-\frac{2}{m}})^n &= \int_{B(p, 2r'_i)}(dd^c\log |s'_i|^{-\frac{2}{m}})^n \\&= \int_{B(p, 2r_i)}(dd^c\log |s_i|^{-\frac{2}{m}})^n + \int_{B(p, 2r'_i)\backslash B(p, 2r_i)}Ric^n \\&= \int_{B(p_i, 2)}(dd^c\log |s_i|^{-\frac{2}{m}})^n + \int_{B(p, 2r'_i)\backslash B(p, 2r_i)}Ric^n.  \end{aligned} 
\end{equation}
Also, under condition $A$, Lemma \ref{highpowerbd} can be proved similarly.
For reader's convenience, let us state it as 

\begin{lemma}\label{highpowerbd1}
Let $M^n$ satisfy condition $A$. Then there exists $C_1>0$ such that for all $1\leq k\leq n$, for all $r>0$, 
$$r^{2k-2n}\int_{B(p, r)}Ric^k\wedge \omega^{n-k}<C_1.$$ Furthermore, if we consider a tangent cone $(V, o)$ and a sequence of rescaled manifolds $(M_i, p_i) = (M, p, \frac{r}{r_i})\to (V, o)$ in the pointed Gromov-Hausdorff sense, then for $k<n$, $$\lim\limits_{i\to\infty}r_i^{2k-2n}\int_{B(p, r_i)}Ric^k\wedge \omega^{n-k}=\int_{B_V(o, 1)}Ric_V^k\wedge \omega_V^{n-k},$$ where $Ric_V$ is understood in the current sense.

\end{lemma}

Then we can show that $$\int_{V}\chi (dd^c\log |s_V|^{-\frac{2}{m}})^n = 0.$$ Then the argument proceeds as before. We can show that all tangent cones are Ricci flat.  Let $X$ be the blow down contraction of exceptional locus for $M$.
Then one can apply \cite{[L2]} to prove that $X$ is isomorphic to an affine algebraic variety in $\mathbb{C}^N$. The rest of the argument is the same as before. 
Uniqueness of tangent cones follows from \cite{[CM]}\cite{[DS]}\cite{[L2]}.
The proof of Theorem \ref{ricflatmain1} is complete. As a side remark, Lemma \ref{highpowerbd1} implies that the condition $\int_M Ric^n<+\infty$ in \cite{[Mok1]} holds automatically (take $k=n$).

\medskip
\medskip

Proof of Corollary \ref{ricflatlocal1}:

We claim that the limit space $(X, o)$ is homeomorphic to a normal complex space. Note $(X, o)$ is complex analytically smooth away from $o$, as $|Rm|\leq \frac{C}{r^2}$. Let $r_0$ be small so that for all $r<r_0$, $B(o, r)$ is Gromov-Hausdorff close to a metric cone, after rescaling the metric so that the radius is $1$. Then we can show $1$-convexity as before. Then, for some $c>0$, we can find a holomorphic embedding from $B(o, cr_0)\to \mathbb{C}^N$,  up to contraction of finitely many compact subvarities to points. Note inside $B(o, cr_0)$, there is no compact variety of positive dimension. Otherwise, one can rescale the metric to the scale of the compact variety, then we find a contradiction by $1$-convexity. This shows $(X, o)$ is homeomorphic to a normal complex space. The rest of the argument is similar to the noncompact case. 
The uniqueness of tangent cone follows from the argument in \cite{[DS]}.

\medskip
\medskip

Proof of Corollary \ref{ricflatlocal2}:
The argument is almost identical to noncompact case. One difference is this: when solving $\overline\partial$-equation, we have to multiply the metric on line bundle by $e^{-Cu}$, where $dd^c u = \omega$. This kills the negative parts of Ric.

\medskip
\medskip

Proof of Corollary \ref{ricflatcorintegral}:

Assume $r_i\to\infty$ such that $\lim\limits_{i\to\infty}r_i^2\dashint_{B(p, r_i)}S = 0$. Define $(M_i, p_i) = (M, p, \frac{r}{r_i})$. We assume $(M_i, p_i)$ converges in the pointed Gromov-Hausdorff sense to $(V, o)$.
Then by applying the convergence result in Lemma \ref{highpowerbd1}, we find $V$ is Ricci flat on the regular part in the current sense. By the regularity theory of complex Monge-Amp\`ere equations, we see that $V$ is Ricci flat in the differential geometric sense. Then apply Theorem \ref{ricflatmain1}.

\medskip
\medskip

Proof of Corollary \ref{ricflatcorpointwise}

This is straightforward, since the tangent cone is Ricci flat.

\medskip
\medskip

Proof of Corollary \ref{ricflatcorintegralsn}

There exists a sequence $r_i\to\infty$, such that $\int_{B(p, 2r_i)\backslash B(p, r_i)} S^n \to 0$. Define $(M_i, p_i) = (M, p, \frac{r}{r_i})$. We assume $(M_i, p_i)$ converges in the pointed Gromov-Hausdorff sense to $(V, o)$.
Then H\"older inequality implies $$\dashint_{B(p_i, 2)\backslash B(p_i, 1)} S_i \leq (\dashint_{B(p_i, 2)\backslash B(p_i, 1)} S^n_i)^{\frac{1}{n}}\to 0.$$ This means $Ric_i$ converges to zero in the current sense on $V$. Then apply Theorem \ref{ricflatmain1}.

\medskip
\medskip

Proof of Corollary \ref{ricflatcorrmn}

This follows from Corollary \ref{ricflatcorintegralsn}.

\end{document}